\documentclass[11pt]{article}

\usepackage{tikz}

\usepackage{amsmath,amsfonts,amsthm}

\usepackage{enumerate}
\newtheorem{defn}{{\bf Definition}}[section]
\newtheorem{eg}[defn]{{\bf Example}} \newtheorem{lemma}[defn]{{\bf
Lemma}} \newtheorem{prop}[defn]{{\bf Proposition}}
 \newtheorem{cor}[defn]{{\bf
Corollary}}

\font\bbb=msbm10 scaled\magstep1

\newcommand{\FF}{\mbox{\bbb F}} 
 
\newcommand{\QQ}{\mbox{\bbb Q}} 
\newcommand{\ZZ}{\mbox{\bbb Z}} 

\font\bbb=msbm8 scaled\magstep1

\newcommand{\TPSS}{S^{\hspace{.2mm}3} \mbox{$\times
\hspace{-2.8mm}_{-}$} \, S^{\hspace{.1mm}1}}

\newcommand{\TPSSD}{S^{\hspace{.2mm}d-1} \mbox{$\times
\hspace{-2.8mm}_{-}$} \, S^{\hspace{.1mm}1}}

 \newcommand{\TPPSS}{\kern.24em \rule width.08em height1.5ex
depth-.08ex \kern-.36em \times}

\newcommand{\Lk}[2]{{\rm Lk}_{#1}(#2)} \newcommand{\st}[2]{{\rm
st}_{#1}(#2)}

\begin{document}

\title{\bf Tight triangulations of some 4-manifolds }
\author{{\bf Basudeb Datta} and {\bf Nitin Singh}}

\date{}

\maketitle

\vspace{-5mm}

\noindent {\small Department of Mathematics, Indian Institute of
Science, Bangalore 560\,012, India.$^1$}

\footnotetext[1]{{\em E-mail addresses:} dattab@math.iisc.ernet.in
(B. Datta), nitin@math.iisc.ernet.in (N. Singh).}

\begin{center}
\date{August 24, 2012}
\end{center}

\hrule

\bigskip

\noindent {\bf Abstract.}

\smallskip

{\small Walkup's class ${\cal K}(d)$ consists of the
$d$-dimensional simplicial complexes all whose vertex links are
stacked $(d-1)$-spheres. According to a result of Walkup, the face
vector of any triangulated 4-manifold $X$ with Euler
characteristic $\chi$ satisfies $f_1 \geq 5f_0 - \frac{15}{2}
\chi$, with equality only for $X \in {\cal K}(4)$. K\"{u}hnel
observed that this implies $f_0(f_0 - 11) \geq -15\chi$, with
equality only for 2-neighborly members of ${\cal K}(4)$. For $n =
6, 11$ and 15, there are triangulated 4-manifolds with $f_0=n$ and
$f_0(f_0 - 11) = -15\chi$. In this article, we present
triangulated 4-manifolds with $f_0 = 21, 26$ and $41$ which
satisfy $f_0(f_0 - 11) = -15\chi$. All these triangulated
manifolds are tight and strongly minimal. }

\bigskip

\noindent {\small {\em MSC 2000\,:} 57Q15, 57R05.


}

\bigskip

\hrule

\section{Introduction}

Walkup's class ${\cal K}(d)$ consists of the $d$-dimensional
simplicial complexes all whose vertex links are stacked
$(d-1)$-spheres. Kalai showed that for $d\geq 4$, all
connected members of ${\cal K}(d)$ are obtained from stacked
$d$-spheres by finitely many elementary handle additions
(cf. Proposition \ref{P3} below).
According to a result of Walkup \cite{wa}, the face vector $(f_0,
f_1, f_2, f_3, f_4)$ of any triangulated 4-manifold $X$ with Euler
characteristic $\chi$ satisfies $f_1 \geq 5f_0 - \frac{15}{2}
\chi$, with equality only for $X \in {\cal K}(4)$. K\"{u}hnel
\cite{ku} observed that this implies $f_0(f_0 - 11) \geq -15\chi$,
with equality only for 2-neighborly members of ${\cal K}(4)$.
Clearly, for the equality, $f_0 \equiv 0, 5, 6, 11$ (mod 15). For
$n = 6, 11$ and 15, there are such triangulated manifolds with
$f_0=n$, namely, the 6-vertex standard 4-sphere $S^{\,4}_6$, the
unique 11-vertex triangulation of $S^{\,3} \times S^1$ of K\"uhnel
and the 15-vertex triangulation of $(\TPSS)^{\#3}$ obtained by
Bagchi and Datta \cite{bd10}. Recently, the second author
\cite{si} found ten 15-vertex triangulations of $(S^{\,3} \times
S^{1})^{\#3}$ and one more 15-vertex triangulation of
$(\TPSS)^{\#3}$.

Observe that if $f_0(f_0 - 11) = -15\chi$ and $f_0\geq 15$ then
$\chi$ is even and negative. Moreover, $-\chi/2$ divides $f_0$ if
and only if $f_0 = 21, 26$ or 41.  Note that, in each of the three
cases, $p=-\chi/2$ is a prime. For these cases, we have
constructed triangulated 4-manifolds which satisfy $f_0(f_0 - 11)
= -15\chi$ and have automorphism groups $\ZZ_p$. More explicitly, we
have constructed a 21-vertex triangulation of $(S^{\,3} \times
S^{1})^{\#8}$, a 21-vertex triangulation of $(\TPSS)^{\#8}$, a
26-vertex triangulation of $(\TPSS)^{\#14}$ and a 41-vertex
triangulation of $(S^{\,3} \times S^{1})^{\#42}$. For each of our
triangulated manifolds, the full automorphism group is $\ZZ_p$,
where $p= -\chi/2$.

Effenberger proved that any 2-neighborly
$\FF$-orientable member of ${\cal K}(4)$ is $\FF$-tight (cf.
Proposition \ref{P4} below). By a result (Proposition \ref{P5}
below) of
Bagchi and Datta, for any field $\FF$, any $\FF$-tight member of ${\cal
K}(4)$ is strongly minimal. Therefore, our orientable (resp.,
non-orientable) examples are $\QQ$-tight (resp., $\ZZ_2$-tight)
and strongly minimal.

\section{Preliminaries}

All simplicial complexes considered here are finite and abstract.
By a triangulated manifold/sphere/ball, we mean an abstract
simplicial complex whose geometric carrier is a topological
manifold/sphere/ball. We identify two complexes if they are
isomorphic. A $d$-dimensional simplicial complex is called {\em
pure} if all its maximal faces (called {\em facets}) are
$d$-dimensional. A $d$-dimensional pure simplicial complex is said
to be a {\em weak pseudomanifold} if each of its $(d - 1)$-faces
is in at most two facets. For a $d$-dimensional weak
pseudomanifold $X$, the {\em boundary} $\partial X$ of $X$ is the
pure subcomplex of $X$ whose facets are those $(d-1)$-dimensional
faces of $X$ which are contained in unique facets of $X$. The {\em
dual graph} $\Lambda(X)$ of a pure simplicial complex $X$ is the
graph whose vertices are the facets of $X$, where two facets are
adjacent in $\Lambda(X)$ if they intersect in a face of
codimension one. A {\em pseudomanifold} is a weak pseudomanifold
with a connected dual graph. All connected triangulated manifolds
are automatically pseudomanifolds.

If $X$ is a $d$-dimensional simplicial complex then, for $0\leq j
\leq d$, the number of its $j$-faces is denoted by $f_j = f_j(X)$.
The vector $f(X) := (f_0, \dots, f_d)$ is called the {\em face
vector} of $X$ and the number $\chi(X) := \sum_{i=0}^{d} (-1)^i
f_i$ is called the {\em Euler characteristic} of $X$. As is well
known, $\chi(X)$ is a topological invariant, i.e., it depends only
on the homeomorphic type of $|X|$. A simplicial complex $X$ is
said to be {\em $l$-neighborly} if any $l$ vertices of $X$ form a
face of $X$. A 2-neighborly simplicial complex is also called a
{\em neighborly} simplicial complex.

A {\em standard $d$-ball} is a pure $d$-dimensional simplicial
complex with one facet. The standard ball with facet $\sigma$ is
denoted by $\overline{\sigma}$. A $d$-dimensional pure simplicial
complex $X$ is called a {\em stacked $d$-ball} if there exists a
sequence $B_1, \dots, B_m$ of pure simplicial complexes such that
$B_1$ is a standard $d$-ball, $B_m=X$ and, for $2\leq i\leq m$,
$B_i = B_{i-1}\cup \overline{\sigma_i}$ and $B_{i-1} \cap
\overline{\sigma_i} = \overline{\tau_i}$, where $\sigma_i$ is a
$d$-face and $\tau_i$ is a $(d-1)$-face of $\sigma_i$. Clearly, a
stacked ball is a pseudomanifold. A simplicial complex is called a
{\em stacked $d$-sphere} if it is the boundary of a stacked
$(d+1)$-ball. A trivial induction on $m$ shows that a stacked
$d$-ball actually triangulates a topological $d$-ball, and hence a
stacked $d$-sphere is a triangulated $d$-sphere. If $X$ is a
stacked ball then clearly $\Lambda(X)$ is a tree. So, a stacked ball
is a pseudomanifold whose dual graph is a tree. But, the converse
is not true (e.g., the 3-pseudomanifold $X$ whose facets are
$1234, 2345, 3456, 4567, 5671$ is a pseudomanifold for which
$\Lambda(X)$ is a tree but $|X|$ is not a ball). Here we have

\begin{lemma}\label{lemma:stackedball} Let $X$ be a pure
$d$-dimensional simplicial complex.  \begin{enumerate} \item[{\rm
(i)}] If $\Lambda(X)$ is a tree then $f_0(X) \leq f_d(X) +d$.
\item[{\rm (ii)}] $\Lambda(X)$ is a tree and $f_0(X) = f_d(X) +d$
if and only if $X$ is a stacked ball.  \end{enumerate}
\end{lemma}

\begin{proof} Let $f_d(X)=m$ and $f_0(X)=n$. So,
$\Lambda(X)$ is a graph with $m$ vertices. We prove (i) by
induction on $m$. If $m=1$ then the result is true with equality.
So, assume that $m >1$ and the result is true for smaller values
of $m$. Since $\Lambda(X)$ is a tree, it has a vertex $\sigma$ of
degree one (leaf) and hence $\Lambda(X)-\sigma$ is again a tree.
Let $Y$ be the pure simplicial complex (of dimension $d$) whose
facets are those of $X$ other than $\sigma$. Since $\sigma$ has a
$(d-1)$-face in $Y$, it follows that $f_0(Y) \geq n-1$. Since
$f_d(Y) = m -1$, the result is true for $Y$ and hence $f_0(Y) \leq
(m -1) +d$. Therefore, $n \leq f_0(Y) +1 \leq 1 + (m-1) +d = m+d$.
This proves (i).

If $X$ is a stacked $d$-ball with $m$ facets then $X$ is a
pseudomanifold and by the definition (since at each of the $m-1$
stages one adds one facet and one vertex), $n = (d+1) + (m-1) =
m+d$. Conversely, let $\Lambda(X)$ be a tree and $n= f_0(X) =
m+d$. Let $Y$ be as above. Since $f_0(Y) \geq n -1$, it follows
that $f_0(Y) = n$ or $n -1$. If $f_0(Y) = n$ then $f_0(Y) = n >
(m-1) +d = f_d(Y) +m$, a contradiction to part (i). So, $f_0(Y) =
n-1$ and hence $Y\cap \overline{\sigma}$ is a $(d-1)$-face of
$\sigma$. Since $f_d(Y) = m-1$, by induction hypothesis, $Y$ is a
stacked $d$-ball and hence $X = Y \cup \overline{\sigma}$ is a
stacked $d$-ball. This
proves (ii). 
\end{proof}

\begin{cor}\label{cor:C0} Let $X$ be a pure $d$-dimensional
simplicial complex and let $CX$ denote a {\em cone} over $X$.
Then $CX$ is a stacked $(d+1)$-ball if and only if $X$ is a
stacked $d$-ball.
\end{cor}

\begin{proof} Notice that $f_{d+1}(CX)=f_d(X)$ and
$f_0(CX)=f_0(X)+1$. Also $\Lambda(CX)$ is naturally isomorphic to
$\Lambda(X)$. The proof now follows from Lemma
\ref{lemma:stackedball}.
\end{proof}

In \cite{wa}, Walkup defined the class ${\cal K}(d)$ as the family
of all $d$-dimensional simplicial complexes all whose vertex-links
are stacked $(d - 1)$-spheres. Clearly, all the members of ${\cal
K}(d)$ are triangulated closed manifolds. Let ${\cal K}^{\ast}(d)$
be the class of 2-neighborly members of  ${\cal K}(d)$. We know
the following.

\begin{prop}[Bagchi and Datta \cite{bd16}]\label{P2}
Let $M$ be a connected closed triangulated manifold of dimension
$d\geq 3$. Let $\beta_1=\beta_1(M;\ZZ_{2})$. Then the face vector
of $M$ satisfies:
\begin{enumerate}[{\rm (a)}]
\item $ f_j \geq \begin{cases}
            \binom{d+1}{j}f_0+j\binom{d+2}{j+1}(\beta_1-1), & \mbox{
if } 1\leq j<d, \\
            df_0+(d-1)(d+2)(\beta_1-1), & \mbox{ if } j=d.
        \end{cases}
$

\item $\binom{f_0-d-1}{2}\geq \binom{d+2}{2}\beta_1$.
\end{enumerate}
When $d\geq 4$, the equality holds in {\rm (a)} $($for some $j
\geq 1$, equivalently, for all $j\,)$ if and only if $M\in {\cal
K}(d)$, and equality holds in {\rm (b)} if and only if $M\in
{\cal K}^{\ast}(d)$.
\end{prop}

The case $d=4$ of the above proposition is due to Walkup \cite{wa}
and K\"{u}hnel \cite{ku}. Part (b) of the above proposition is due
to Lutz, Sulanke and Swartz \cite{LSS}.

\begin{prop}[Kalai \cite{ka}]\label{P3} For $d\geq 4$, a connected
simplicial complex $X$ is in ${\cal K}(d)$ if and only if $X$ is
obtained from a stacked $d$-sphere by $\beta_1(X)$ combinatorial
handle additions. In consequence, any such $X$ triangulates either
$(S^{\,d -1}\!\times S^1)^{\# \beta_1}$ or $(\TPSSD)^{\#\beta_1}$
according as $X$ is orientable or not. $($Here $\beta_1 =
\beta_1(X).)$ \end{prop}

It follows from Proposition \ref{P3}  that \begin{eqnarray}
\label{eq:beta1} \chi(X) = 2 - 2 \beta_1(X) \, \mbox{ for } \,
X\in {\cal K}(d).  \end{eqnarray}

For a field $\FF$, a $d$-dimensional simplicial complex $X$ is
called {\em tight with respect to} $\FF$ (or {\em $\FF$-tight}) if
(i) $X$ is connected, and (ii) for all induced subcomplexes $Y$ of
$X$ and for all $0\leq j \leq d$, the morphism $H_{j}(Y; \FF) \to
H_{j}(X; \FF)$ induced by the inclusion map $Y \hookrightarrow X$
is injective. If $X$ is $\QQ$-tight then it is $\FF$-tight for all
fields $\FF$ and called {\em tight} (cf. \cite{bd17}).

A $d$-dimensional simplicial complex $X$ is called {\em minimal}
if $f_0(X) \leq f_0(Y)$ for every triangulation $Y$ of the
geometric carrier $|X|$ of $X$. We say that $X$ is {\em strongly
minimal} if $f_i(X) \leq f_i(Y)$, $0\leq i \leq d$, for all such
$Y$. We know the following.

\begin{prop}[Effenberger \cite{ef}, Bagchi and Datta
\cite{bd16}]\label{P4} Every $\FF$-orientable member of ${\cal
K}^{\ast}(d)$ is $\FF$-tight for $d\neq 3$. An $\FF$-orientable
member of ${\cal K}^{\ast}(3)$ is $\FF$-tight if and only if
$\beta_1(X)=(f_0(X)-4)(f_0(X)-5)/20$.
\end{prop}

\begin{prop}[Bagchi and Datta \cite{bd16}]\label{P5} Every
$\FF$-tight member of\,  ${\cal K}(d)$ is strongly minimal.
\end{prop}

Let $\overline{{\cal K}}(d)$ be the class of all $d$-dimensional
simplicial complexes all whose vertex-links are stacked
$(d-1)$-balls. Clearly, if $N \in \overline{{\cal K}}(d)$ then $N$
is a triangulated manifold with boundary and satisfies
\begin{eqnarray} \label{eq:skel} {\rm skel}_{d-2}(N) = {\rm
skel}_{d-2}(\partial N).  \end{eqnarray} Here ${\rm skel}_{j}(N) =
\{\alpha\in N \, : \, \dim(\alpha) \leq j\}$ is the $j$-skeleton
of $N$. We know the following.

\begin{prop}[Bagchi and Datta \cite{bd17}]\label{P6}
For $d \geq 4$, $M\mapsto \partial M$ is a bijection from
$\overline{\cal K}(d+1)$ to ${\cal K}(d)$.
\end{prop}

\begin{cor}\label{cor:C1}
For $d\geq 4$, if $M\in \overline{\cal K}(d+1)$ then ${\rm
Aut}(M)={\rm Aut}(\partial M)$.
\end{cor}

\begin{proof}
Clearly ${\rm Aut}(M)\subseteq {\rm Aut}(\partial M)$. If
$\sigma:V(M)\rightarrow V(M)$ is in ${\rm Aut}(\partial M)$ then
$\sigma(M)\in \overline{\cal K}(d+1)$ and
$\partial(\sigma(M))=\sigma(\partial M)=\partial M$. Therefore by
Proposition \ref{P6}, $\sigma(M)=M$. This implies $\sigma\in {\rm
Aut}(M)$. Therefore, ${\rm Aut}(\partial M)\subseteq {\rm Aut}(M)$
and hence ${\rm Aut}(M)={\rm Aut}(\partial M)$.
\end{proof}

\section{Examples}

\begin{eg}\label{E1} {\rm Let
$V_{21}=\cup_{i=0}^6\{a_i,b_i,c_i\}$ be a set of $21$ elements.
Let the cyclic group $\ZZ_7$ act on $V_{21}$ as $i\cdot a_j =
a_{i+j}$, $i\cdot b_j=b_{i+j}$ and $i\cdot c_j=c_{i+j}$ (additions
being modulo $7$). Consider the pure $5$-dimensional simplicial
complex $A^{5}_{21}$ on the vertex-set $V_{21}$ as follows. Modulo
the group $\ZZ_7$ the facets are
\begin{eqnarray*}
&&\sigma_{0}=a_0a_1a_2b_0b_1c_0, \kappa_{0}=a_1a_2b_0b_1b_2c_0,
\tau_{0}=a_1a_2a_3b_0b_1b_2, \alpha_{0}=a_0a_1b_0b_1c_0c_3, \\
&&\beta_{0}=a_0a_1b_0b_3c_0c_3, \mu_{0}=a_0b_0b_3c_0c_3c_4,
\nu_{0}=a_0a_3b_3c_0c_3c_4, \gamma_{0}=a_3b_3c_0c_3c_4c_6.
\end{eqnarray*}
The full list of $56$ facets can be obtained by applying the group
$\ZZ_7$ to these eight facets. The dual graph of $A^{5}_{21}$ is
the union of two $21$-cycles $C_1=\sigma_0\kappa_0\tau_0 \sigma_1
\kappa_1 \tau_1\cdots \sigma_6\kappa_6\tau_6\sigma_0$, $C_2=
\mu_{0}\nu_{0}\gamma_{0}\mu_{3}\nu_{3}\gamma_{3}\cdots \mu_{4}
\nu_{4} \gamma_{4}\mu_{0}$ and paths $P_i= \sigma_{i}\alpha_{i}
\beta_{i}\mu_{i}$ for $i\in \ZZ_7$. It can be shown that
$A^{5}_{21}$ is a neighborly member of $\overline{\mathcal K}(5)$
(see Lemma \ref{lemma:N411} below). Let $M^{4}_{21} := \partial
A^{5}_{21}$. Then $M^{4}_{21}\in {\cal K}^{\ast}(4)$ and hence, by
Proposition \ref{P2}, $\chi(M^{4}_{21})=-14$. Then by
(\ref{eq:beta1}), $\beta_1(M^{4}_{21})=8$. One can show that
$M^{4}_{21}$ is orientable (by giving an explicit orientation or
using {\tt simpcomp} \cite{simpcomp}) and so, by Proposition
\ref{P3}, $M^{4}_{21}$ triangulates $(S^3\times S^1)^{\#8}$. }
\end{eg}

\begin{eg}\label{E2} {\rm Let $V_{21}$ be the vertex-set with
group $\ZZ_7$ acting on it as in Example \ref{E1}. Consider the
pure $5$-dimensional simplicial complex $B^{5}_{21}$ whose facets
modulo $\ZZ_7$ action described above are
\begin{eqnarray*}
&&\sigma_{0}=a_0a_1a_2b_0b_1c_0, \, \kappa_{0}=a_0a_1a_2b_1b_2c_0,
\,  \tau_{0}=a_0a_1a_2a_3b_1b_2, \,
\alpha_{0}=a_0a_1b_0b_1c_0c_3, \\
&&\beta_{0}=a_0b_0b_1b_3c_0c_3, \, \mu_{0}=a_0b_0b_3c_0c_3c_4, \,
\nu_{0}=a_3b_0b_3c_0c_3c_4, \, \gamma_{0}=a_3b_3c_0c_3c_4c_6.
\end{eqnarray*}
The dual graph of $B^{5}_{21}$ is the same as that of
$A^{5}_{21}$. It can be shown that $B^{5}_{21}$ is a neighborly
member of $\overline{\mathcal K}(5)$ (see Lemma \ref{lemma:N411}
below). Let $N^{4}_{21} := \partial B^{5}_{21}$. Then
$N^{4}_{21}\in {\cal K}^{\ast}(4)$ and hence, by Proposition
\ref{P2}, $\chi(N^{4}_{21}) =-14$. Then by (\ref{eq:beta1}),
$\beta_1(N^{4}_{21}) =8$. Using {\tt simpcomp}, one can check that
$N^{4}_{21}$ is non-orientable and so, by Proposition \ref{P3}, it
triangulates $({\TPSS})^{\#8}$. }
\end{eg}

\begin{eg}\label{E3} {\rm Let
$V_{26}=\cup_{i=0}^{12}\{a_i,b_i\}$ be a set of $26$ elements. The
cyclic group $\ZZ_{13}$ acts on $V_{26}$ as $i\cdot a_j=a_{i+j}$,
$i\cdot b_j=b_{i+j}$ (additions being modulo 13). Consider the
$5$-dimensional pure simplicial complex $B^{5}_{26}$ on the
vertex-set $V_{26}$ whose facets modulo the group $\ZZ_{13}$ are
\begin{eqnarray*}
&&\sigma_{0}=a_0a_{10}a_{11}a_{12}b_9b_{10}, \,
\tau_{0}=a_0a_1a_{10}a_{11}a_{12}b_{10}, \,
\alpha_{0}=a_0a_{11}a_{12}b_5b_9b_{10}, \\
&&\beta_{0}=a_0a_{11}a_{12}b_2b_5b_{10}, \,
\gamma_{0}=a_0a_7a_{12}b_2b_5b_{10}, \,
\mu_{0}=a_7a_{12}b_0b_2b_5b_{10}, \,
\delta_{0}=a_7b_0b_2b_5b_8b_{10}.
\end{eqnarray*}
The full list of $91$ facets can be obtained by applying the group
$\ZZ_{13}$ to these seven facets. The dual graph of $B^{5}_{26}$
is the union of two $26$-cycles $C_1=\sigma_{0}\tau_{0}
\sigma_{1}\tau_{1} \cdots\sigma_{12}\tau_{12}\sigma_{0}$,
$C_2=\mu_{0}\delta_{0}\mu_{8}\delta_{8}\cdots\mu_{5}\delta_{5}\mu_{0}$
and paths $P_i=\sigma_i\alpha_i\beta_i\gamma_i\mu_i$ for $i\in
\ZZ_{13}$. It can be shown that $B^{5}_{26}$ is a neighborly
member of $\overline{\mathcal K}(5)$ (see Lemma \ref{lemma:N411}
below). Let $N^{4}_{26} := \partial B^{5}_{26}$. Then
$N^{4}_{26}\in {\cal K}^{\ast}(4)$ and hence, by Proposition
\ref{P2}, $\chi(N^{4}_{26})=-26$. Then by (\ref{eq:beta1}),
$\beta_1(N^{4}_{26})=14$. Using {\tt simpcomp}, one can check that $N^{4}_{26}$ is
non-orientable and so, by Proposition \ref{P3}, $N^{4}_{26}$
triangulates $(\TPSS)^{\#14}$. }
\end{eg}

\begin{eg}\label{E4} {\rm
Let $V_{41} = \{a_0, a_1,\dots, a_{40}\}$ be a set of 41 elements.
The cyclic group $\ZZ_{41}$ acts on $V_{41}$ as $i\cdot a_j =
a_{i+j}$ (addition is modulo 41).
Consider the pure 5-dimension simplicial complex $A^{5}_{41}$
on the vertex-set $V_{41}$ as follows. Modulo the group $\ZZ_{41}$
its facets are
\begin{eqnarray*}
&&\sigma_{0} = a_{36} a_{37} a_{38} a_{39} a_{40}a_{0}, \,
\alpha_{0} =  a_{36} a_{37} a_{38} a_{39} a_{0} a_{6}, \,
\beta_{0} =  a_{37} a_{38} a_{39} a_{0} a_{6}a_{13}, \\
&& \gamma_{0} = a_{38} a_{39} a_{0} a_{6}a_{13}a_{20}, \,
\delta_{0} =a_{39} a_{0} a_{6}a_{13}a_{20}a_{27}, \, \mu_{0} =
a_{6}a_{13}a_{20}a_{27}a_{34}a_{0}.
\end{eqnarray*}
The full list of 246 facets of $A^{5}_{41}$ may be obtained from
these basic six facets applying the group $\ZZ_{41}$. The dual
graph of $A^{5}_{41}$  is the union of two $41$-cycles $C_1=
\sigma_{0}\sigma_{1}\cdots\sigma_{40}\sigma_{0}$, $C_2= \mu_{0}
\mu_{7}\mu_{14}\cdots\mu_{34}\mu_{0}$ and paths $P_i= \sigma_{i}
\alpha_{i}\beta_{i} \gamma_{i}\delta_{i}\mu_{i}$ for $i\in
\ZZ_{41}$. Then $A^{5}_{41}$ is a neighborly member of
$\overline{{\cal K}}(5)$ (see Lemma \ref{lemma:N411} below). Let
$M^{4}_{41} := \partial A^{5}_{41}$. Then  $M^{4}_{41} \in {\cal
K}^{\ast}(4)$ and hence, by Proposition \ref{P2},
$\chi(M^{4}_{41}) = -82$. Therefore, by (\ref{eq:beta1}),
$\beta_1(M^{4}_{41}) = 1 - \chi(M^{4}_{41})/2 = 42$. One can check
(by giving an explicit orientation or
using {\tt simpcomp}) that $M^{4}_{41}$ is orientable and hence, by Proposition
\ref{P3}, $M^{4}_{41}$ triangulates $(S^{\,3}\!\times
S^1)^{\#42}$.
 }
\end{eg}

For easy reference, we summarize the results of this section in
table below. Notice that $M^{4}_{41}$ admits a vertex-transitive
automorphism group.
\begin{table}[h]
\centering
\begin{tabular}{|l|c|c|c|c|c|c|}
\hline
&&&&&&\\[-4mm]
$M$ & $f_0(M)$ & $\chi(M)$
& $\beta_1(M)$ & $\mathrm{Aut}(M)$ & $f(M)$ & $|M|$ \\[0.5mm]
\hline
&&&&&&\\[-4mm]
$M^{4}_{21}$ & $21$ & $-14$ & $8$ & $\ZZ_7$ &
$(21,210,490,525,210)$ &
$(S^3\times S^1)^{\#8}$ \\[0.5mm]
$N^{4}_{21}$ & $21$ & $-14$ & $8$ & $\ZZ_7$ &
$(21,210,490,525,210)$ & $({\TPSS})^{\#8}$ \\[0.5mm]
$N^{4}_{26}$ & $26$ & $-26$
& $14$ & $\ZZ_{13}$ & $(26,325,780,845,338)$ & $(\TPSS)^{\#14}$ \\[0.5mm]
$M^{4}_{41}$ & $41$ & $-82$ & $42$ & $\ZZ_{41}$ &
$(41,820,2050,2255,902)$ & $(S^3\times S^1)^{\#42}$ \\[0.5mm]
\hline
\end{tabular}
\caption{Summary of results of Section 3} \label{tbl:tbl1}
\end{table}

\section{Construction Details}

Let $X$ be a neighborly member of $\overline{\cal K}(d)$. Then all
vertex-links, and equivalently vertex-stars in $X$ are stacked
balls. By Corollary \ref{cor:C0}, we see that the facets
containing a given vertex $x$ form an $(f_0(X)-d)$-vertex induced
subtree of $\Lambda(X)$. Thus for each vertex, we get a subtree of
$\Lambda(X)$ (namely, the dual graph of $\st{X}{x}$). From the
neighborliness of $X$, it follows that any two of these trees
intersect. Now we invert the question, i.e, given a graph $G$ and
an intersecting family $\cal T$ of induced subtrees of $G$, can we
get a neighborly member of $\overline{\cal K}(d)$? Our next lemma
answers this in affirmative under certain conditions. Given a
graph $G$ and a family ${\cal T}=\{T_i\}_{i\in {\cal I}}$ of
induced subtrees of $G$, we say that $u \in V(G)$ {\em defines}
the subset $\hat{u} =\{i \in {\cal I} : u \in V(T_i)\}$ of $\cal
I$.

\begin{lemma}\label{lemma:construction}
Let $G$ be a graph and ${\cal T} = \{T_i\}_{i=1}^{n}$ be a  family
of $(n-d)$-vertex induced subtrees of $G$, any two of which
intersect. Suppose that {\rm (i)} each vertex of $G$ is in exactly
$d+1$ members of $\cal T$  and {\rm (ii)} for any two vertices
$u\neq v$ of $G$, $u$ and $v$ are together in exactly $d$ members
of $\cal T$ if and only if $uv$ is an edge of $G$. Then the pure
simplicial complex $M$ whose facets are $\hat{u}$, where $u \in
V(G)$, is a neighborly member of $\overline{\cal K}(d)$, with
$\Lambda(M)\cong G$.
\end{lemma}

\begin{proof}
Let ${S} \subseteq \{1, \dots, n\}$ be of size $d$. We show that
at most two facets of $M$ contain $S$. If possible, let $\hat{u}$,
$\hat{v}$ and $\hat{w}$ be three facets of $M$ that contain $S$.
Then by assumption, $uv$, $uw$ and $vw$ are edges in $G$. Let
$i\in {S}$. Then by definition of $M$, $u$, $v$, $w$ are vertices
of $T_i$. Since $T_i$ is induced subgraph, we conclude that $uv$,
$uw$, $vw$ are edges of $T_i$, which is a contradiction to the
fact that $T_i$ is a tree. Thus $M$ is a $d$-dimensional weak
pseudomanifold. Clearly $u\mapsto \hat{u}$ is an isomorphism
between $G$ and $\Lambda(M)$. Further the conditions on $(G, \cal
T)$ imply that $G$ should be connected. Thus $M$ is a
$d$-pseudomanifold. Since any two members of $\cal T$ intersect,
it follows that $M$ is neighborly. Let $S_i = \st{M}{i}$ be the
star of the vertex $i$ in $M$. Then by construction $\Lambda(S_i)
\cong T_i$ and thus $f_d(S_i) = \#(V(T_i)) = n-d$. Also from the
neighborliness of $M$, $f_0(S_i)=n$. Thus $f_0(S_i)=f_d(S_i)+d$
and hence, by Lemma \ref{lemma:stackedball}, $S_i$ is a stacked
$d$-ball. Therefore, by Corollary \ref{cor:C0}, $\Lk{M}{i}$ is a
stacked $(d-1)$-ball and hence $M$ is a member of $\overline{\cal
K}(d)$.
\end{proof}

We use Lemma \ref{lemma:construction} to construct all the
complexes. Here we present the details of the construction of
$A^{5}_{41}$ and $M^{4}_{41}=\partial A^{5}_{41}$.

\medskip

\noindent {\bf Construction of \boldmath{$A^{5}_{41}$}:} Let $G$
be the union of two $41$-cycles $C_1=u_0u_1\cdots u_{40} u_0$,
$C_2= v_0 v_7 v_{14}\cdots  v_{34}  v_0$ and the paths $P_i=u_i
x_i y_i z_i  w_i  v_i$ for $i\in \ZZ_{41}$. Consider the family of
induced subtrees of $G$ defined by ${\cal T}=\{T_i\}_{i=0}^{40}$,
where $T_i$ is the subtree induced on $G$ by the following $36$
vertices (see Fig \ref{fig:Ti}):
\begin{align*}
u_{i},u_{i+1},\ldots,u_{i+5},  v_{i}, v_{i+7},\ldots, v_{i+35},
 x_{i}, y_{i}, z_{i}, w_{i},
 x_{i+2}, y_{i+2}, z_{i+2}, w_{i+2},
 x_{i+3}, y_{i+3}, \\ z_{i+3},  x_{i+4}, y_{i+4},
 x_{i+5}, w_{i+14},  w_{i+21}, z_{i+21},
 w_{i+28}, z_{i+28}, y_{i+28},  w_{i+35}, z_{i+35},
y_{i+35}, x_{i+35}.
\end{align*}

We show that $(G,{\cal T})$ satisfy the conditions in Lemma
\ref{lemma:construction} for $d=5$. From Figure \ref{fig:Ti}, it
is easily observed that for $i\in \ZZ_{41}$, {\small
\begin{align*} \hat{u}_i=\{i,i-1,i-2,i-3,i-4,i-5\},
\quad & \hat{ x}_i=\{i,i-2,i-3,i-4,i-5,i-35\}, \\
\hat{ y}_i=\{i,i-2,i-3,i-4,i-28,i-35\}, \quad &
\hat{ z}_i=\{i,i-2,i-3,i-21,i-28,i-35\}, \\
\hat{ w}_i=\{i,i-2,i-14,i-21,i-28,i-35\}, \quad & \hat{
v}_i=\{i,i-7,i-14,i-21,i-28,i-35\}.
\end{align*} }
Clearly each vertex of $G$ defines a $6$-subset. Further it can be
seen that $\hat{x}\cap \hat{y}$ is a $5$-element set only for edge
pairs like $(\hat{u}_i,\hat{u}_{i+1})$, $(\hat{ v}_i, \hat{ v}_{i
+7})$, $(\hat{u}_i, \hat{ x}_i)$, $(\hat{ x}_i, \hat{ y}_i)$ etc.
Now we show that $\cal T$ is an intersecting family. First we
notice that
$$
\varphi := (u_0\cdots u_{40})( x_0\cdots  x_{40}) ( y_0\cdots
y_{40})( z_0\cdots  z_{40})( w_0\cdots  w_{40})( v_0\cdots v_{40})
$$
is an automorphism of $G$ and further $\varphi(T_i)=T_{i+1}$ for
$i\in \ZZ_{41}$. Thus we have $T_i=\varphi^i(T_0)$, and so to
prove $\cal T$ to be an intersecting family, it is sufficient to
prove that $T_0$ has non-empty intersection with $T_1,\ldots,
T_{20}$. Clearly $T_1,\ldots,T_5$ intersect $T_0$ in $ u_1,\ldots,
u_5$ respectively; $T_7,T_{14}$ intersect $T_0$ in $ v_7, v_{14}$
respectively. Since $6+35=13+28=20+21=0\ (\text{mod } 41)$, we see
that $T_6$, $T_{13}$, $T_{20}$ intersect $T_0$ in $ v_0$. Since
$8+35=2\ (\text{mod } 41)$ we see that $T_8$ contains $ x_2$,
which also appears in $T_0$. Similarly, $ x_3 \in T_{9}\cap
T_{0}$, $ x_4 \in T_{10}\cap T_{0}$, $ x_5 \in T_{11}\cap T_{0}$,
$ w_{14} \in T_{12}\cap T_{0}$, $ y_2 \in T_{15}\cap T_{0}$, $ y_3
\in T_{16}\cap T_{0}$, $ y_4 \in T_{17}\cap T_{0}$, $ z_{21} \in
T_{18}\cap T_{0}$ and $ z_{21} \in T_{19}\cap T_{0}$. Thus, via
construction in Lemma \ref{lemma:construction} , $(G,{\cal T})$
yields a neighborly member of $\overline{\cal K}(5)$, which we
denote by $A^{5}_{41}$. Finally we note that $\pi \colon i \mapsto
i+1$ is an automorphism of $A^{5}_{41}$ by noticing that
$\pi(\hat{ u}_i)= \hat{ u}_{i+1}$, $\pi(\hat{ x}_i)=\hat{
x}_{i+1}$ etc. This generates the automorphism group $\ZZ_{41}$ of
$A^{5}_{41}$, which indeed is the full automorphism group of
$A^{5}_{41}$ (checked by {\tt simpcomp}).
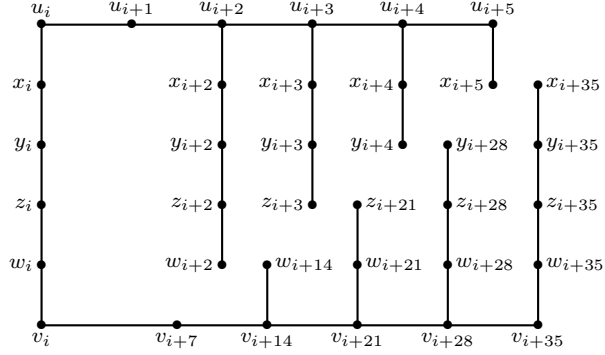
\begin{figure}[tp]
\centering
\begin{tikzpicture}[xscale=1.2,yscale=0.8]
\coordinate (V0) at (0,0); \coordinate (V7) at (1.5,0);
\coordinate (V14) at (2.5,0); \coordinate (V21) at (3.5,0);
\coordinate (V28) at (4.5,0); \coordinate (V35) at (5.5,0);
\coordinate (U0) at (0,5); \coordinate (U1) at (1,5); \coordinate
(U2) at (2,5); \coordinate (U3) at (3,5); \coordinate (U4) at
(4,5); \coordinate (U5) at (5,5); \coordinate (A2) at (2,4);
\coordinate (B2) at (2,3); \coordinate (C2) at (2,2); \coordinate
(D2) at (2,1); \coordinate (A3) at (3,4); \coordinate (B3) at
(3,3); \coordinate (C3) at (3,2); \coordinate (A4) at (4,4);
\coordinate (B4) at (4,3); \coordinate (A5) at (5,4); \coordinate
(D14) at (2.5,1); \coordinate (D21) at (3.5,1); \coordinate (C21)
at (3.5,2); \coordinate (D28) at (4.5,1); \coordinate (C28) at
(4.5,2); \coordinate (B28) at (4.5,3); \coordinate (D35) at
(5.5,1); \coordinate (C35) at (5.5,2); \coordinate (B35) at
(5.5,3); \coordinate (A35) at (5.5,4); \coordinate (A0) at (0,4);
\coordinate (B0) at (0,3); \coordinate (C0) at (0,2); \coordinate
(D0) at (0,1);

\draw [thick] (V0) -- (V35); \draw [thick] (U0) -- (U5); \draw
[thick] (U0) -- (V0); \draw [thick] (U2) -- (D2); \draw [thick]
(U3) -- (C3); \draw [thick] (U4) -- (B4); \draw [thick] (U5) --
(A5); \draw [thick] (V14) -- (D14); \draw [thick] (V21) -- (C21);
\draw [thick] (V28) -- (B28); \draw [thick] (V35) -- (A35);

{\scriptsize \node [above] at (U0) {$u_i$}; \node [above] at (U1)
{$u_{i+1}$}; \node [above] at (U2) {$u_{i+2}$}; \node [above] at
(U3) {$u_{i+3}$}; \node [above] at (U4) {$u_{i+4}$}; \node [above]
at (U5) {$u_{i+5}$};

\node [below] at (V0) {$ v_i$}; \node [below] at (V7) {$
v_{i+7}$}; \node [below] at (V14) {$ v_{i+14}$}; \node [below] at
(V21) {$ v_{i+21}$}; \node [below] at (V28) {$ v_{i+28}$}; \node
[below] at (V35) {$ v_{i+35}$}; \node [left] at (A0) {$ x_i$};
\node [left] at (B0) {$ y_i$}; \node [left] at (C0) {$ z_i$};
\node [left] at (D0) {$ w_i$};

\node [left] at (A2) {$ x_{i+2}$}; \node [left] at (B2) {$
y_{i+2}$}; \node [left] at (C2) {$ z_{i+2}$}; \node [left] at (D2)
{$ w_{i+2}$}; \node [left] at (A3) {$ x_{i+3}$}; \node [left] at
(B3) {$ y_{i+3}$}; \node [left] at (C3) {$ z_{i+3}$}; \node [left]
at (A4) {$ x_{i+4}$}; \node [left] at (B4) {$ y_{i+4}$}; \node
[left] at (A5) {$ x_{i+5}$};

\node [right] at (D14) {$ w_{i+14}$}; \node [right] at (D21) {$
w_{i+21}$}; \node [right] at (C21) {$ z_{i+21}$}; \node [right] at
(D28) {$ w_{i+28}$}; \node [right] at (C28) {$ z_{i+28}$}; \node
[right] at (B28) {$ y_{i+28}$}; \node [right] at (D35) {$
w_{i+35}$}; \node [right] at (C35) {$ z_{i+35}$}; \node [right] at
(B35) {$ y_{i+35}$}; \node [right] at (A35) {$ x_{i+35}$};

\node at (U0) {$\bullet$}; \node at (U1) {$\bullet$};

\node at (U2) {$\bullet$}; \node at  (U3) {$\bullet$};

\node at (U4) {$\bullet$}; \node at  (U5) {$\bullet$};

\node at (V0) {$\bullet$}; \node at  (V7) {$\bullet$};

\node at (V14) {$\bullet$}; \node at  (V21) {$\bullet$};

\node at (V28) {$\bullet$}; \node at  (V35) {$\bullet$};

\node at (A0) {$\bullet$}; \node at  (B0) {$\bullet$};

\node at (C0) {$\bullet$}; \node at  (D0) {$\bullet$};

\node at (A2) {$\bullet$}; \node at  (B2) {$\bullet$};

\node at (C2) {$\bullet$}; \node at  (D2) {$\bullet$};

\node at (A3) {$\bullet$}; \node at  (B3) {$\bullet$};

\node at (C3) {$\bullet$}; \node at  (A4) {$\bullet$};

\node at (B4) {$\bullet$}; \node at  (A5) {$\bullet$};

\node at (D14) {$\bullet$}; \node at  (D21) {$\bullet$};

\node at (C21) {$\bullet$}; \node at  (D28) {$\bullet$};

\node at (C28) {$\bullet$}; \node at  (B28) {$\bullet$};

\node at (D35) {$\bullet$}; \node at  (C35) {$\bullet$};

\node at (B35) {$\bullet$}; \node at  (A35) {$\bullet$};

}
\end{tikzpicture}
\caption{Tree $T_i$ in $G \cong \Lambda(A^{5}_{41})$}
\label{fig:Ti}
\end{figure}

\begin{lemma}\label{lemma:N411}
Let $A^{5}_{21}$, $B^{5}_{21}$, $B^{5}_{26}$, $A^{5}_{41}$,
$M^{4}_{21}$, $N^{4}_{21}$, $N^{4}_{26}$ and $M^{4}_{41}$ be as in
Section $3$. Then
\begin{enumerate}[{\rm (a)}]
\item $A^{5}_{21}, B^{5}_{21}, B^{5}_{26}, A^{5}_{41} \in
\overline{{\cal K}}(5)$, \item ${\rm Aut}(A^{5}_{21})={\rm
Aut}(M^{4}_{21})={\rm Aut}(B^{5}_{21})={\rm Aut}(N^{4}_{21})=\ZZ_7$,
\item ${\rm Aut}(B^{5}_{26})={\rm Aut}(N^{4}_{26})=\ZZ_{13}$, \item
${\rm Aut}(A^{5}_{41}) = {\rm Aut}(M^{4}_{41}) = \ZZ_{41}$.
\end{enumerate}
\end{lemma}

\begin{proof} The properties of the complexes follow from the
constructions. As a prototype, we described the construction of
$A^{5}_{41}$. The properties of other complexes, mentioned in the
statement of the lemma and in Table \ref{tbl:tbl1} may be verified
by using a combinatorial topology package such as {\tt simpcomp}
\cite{simpcomp}. For sake of brevity, we omit all the details
here.
\end{proof}

\begin{lemma}\label{lemma:tightness}
Let $M^{4}_{21}$, $N^{4}_{21}$, $N^{4}_{26}$ and $M^{4}_{41}$ be
as in Section $3$. Then
\begin{enumerate}[{\rm (a)}]
\item $M^{4}_{21}$ and $M^{4}_{41}$ are $\mathbb{Q}$-tight.
\item $N^{4}_{21}$ and $N^{4}_{26}$ are $\ZZ_2$-tight.
\item
$M^{4}_{21}$, $N^{4}_{21}$, $N^{4}_{26}$ and $M^{4}_{41}$ are
strongly minimal.
\end{enumerate}
\end{lemma}

\begin{proof}
As previously seen $M^{4}_{21}$ and $M^{4}_{41}$ are
triangulations of $(S^3\times S^1)^{\#8}$
and $(S^3\times S^1)^{\#42}$ respectively and are
in ${\cal K}^{\ast}(4)$. By Proposition \ref{P4}, they are
$\mathbb{Q}$-tight.
Similarly $N^{4}_{21}$, $N^{4}_{26}$ are
triangulations of $({\TPSS})^{\#8}$ and $({\TPSS})^{\#14}$
respectively and are in ${\cal K}^{\ast}(4)$. By Proposition
\ref{P4}, they are $\ZZ_2$-tight. By Proposition \ref{P5}, all the
complexes here are strongly minimal.
\end{proof}


\noindent {\bf Acknowledgement:} The authors thank Bhaskar Bagchi
for some valuable comments on an initial version of this article.
The authors thank UGC Centre for Advanced Study for financial
support. The second author thanks IISc Mathematics Initiative for
support.


\end{document}